\documentclass[a4paper,12pt,leqno]{amsart}

\usepackage{graphicx}
\usepackage{amsmath}
\usepackage{amsfonts}
\usepackage{amssymb}
\usepackage{mathtools}
\usepackage[all,cmtip]{xy}
\usepackage{amsthm}
\usepackage{comment}
\usepackage{enumitem}
\usepackage{url}
\usepackage{epsfig}
\usepackage[utf8]{inputenc}
\usepackage[colorinlistoftodos]{todonotes}
\usepackage[capitalise]{cleveref}
\usepackage{geometry}
\makeatletter
\providecommand\@dotsep{5}
\def\listtodoname{List of Todos}
\def\listoftodos{\@starttoc{tdo}\listtodoname}
\makeatother

\newtheorem{theorem}{Theorem}[section]
\newtheorem{proposition}[theorem]{Proposition}
\newtheorem{corollary}[theorem]{Corollary}
\newtheorem{lemma}[theorem]{Lemma}

  \theoremstyle{definition}

\newtheorem{remark}[theorem]{Remark}
\newtheorem{question}[theorem]{Question}

\newcommand{\dbN}{\mathbb{N}}

\newcommand{\dbR}{\mathbb{R}}

\newcommand{\dbZ}{\mathbb{Z}}

\newcommand{\Z}{\mathbb{Z}}

\newcommand{\calF}{{\mathcal F}}

\newcommand{\calN}{{\mathcal N}}
\newcommand{\calO}{{\mathcal O}}

\newcommand{\orf}[1]{O_\calF #1}

\newcommand{\nbeq}{\begin{equation}}
\newcommand{\neeq}{\end{equation}}
\newcommand{\beq}{\begin{equation*}}
\newcommand{\eeq}{\end{equation*}}

\DeclareMathOperator{\cd}{cd}
\DeclareMathOperator{\vcd}{vcd}
\DeclareMathOperator{\cdfin}{\underline{cd}}

\DeclareMathOperator{\gdfin}{\underline{gd}}

\DeclareMathOperator{\mcg}{Mcg}

\newcommand{\Ngnb}{N_{g,n}^{b}}
\newcommand{\calNgnb}{\calN_{g,n}^{b}}
\newcommand{\Sgnb}{S_{g,n}^{b}}
\newcommand{\Gammagnb}{\Gamma_{g,n}^{b}}
\newcommand{\Gammagn}{\Gamma_{g,n}}

\newcommand{\Ngn}[3][ ]{N_{#2,#3}^{#1}}
\newcommand{\calNgn}[3][ ]{\calN_{#2,#3}^{#1}}
\newcommand{\calSgn}[3][ ]{\Gamma_{#2,#3}^{#1}}
\newcommand{\vcdNgoverF} {\frac{\vcd(\calN_g)+1}{|F|}}

\DeclareMathOperator{\Ker}{Ker}
\DeclareMathOperator{\Ima}{Im}

\begin{document}

\title[On the dimensions of non-orientable mapping class group]{On the dimensions of mapping class groups of non-orientable surfaces}




\author[C. E. Hidber]{Cristhian E. Hidber}
\address{Centro de Ciencias Matemáticas, UNAM Universidad Nacional Autónoma de México, Morelia, Mich. 58190, Mexico}
\email{hidber@matmor.unam.mx	}

\author[L. J. Sánchez Saldaña]{Luis Jorge S\'anchez Salda\~na}
\address{Departamento de Matemáticas, Facultad de Ciencias, Universidad Nacional Autónoma de México, Circuito Exterior S/N, Cd. Universitaria, Colonia Copilco el Bajo, Delegación Coyoacán, 04510, México D.F., Mexico}
\email{luisjorge@ciencias.unam.mx}

\author[A. Trujillo-Negrete]{Alejandra Trujillo-Negrete}
\address{Centro de Investigaci\'on en Matem\'aticas, A. C.  Jalisco S/N, Col. Valenciana CP: 36023 Guanajuato, Gto, México}
\email{alejandra.trujillo@cimat.mx}

\subjclass[2010]{Primary 20F34, 20J05, 20F65}

\date{}


\keywords{Mapping class group, Non-orientable surface, virtual cohomological dimension, proper cohomological dimension, proper geometric  dimension}

\begin{abstract} Let $\calN_g$ be the mapping class group of a non-orientable closed surface.
We prove that the proper cohomological dimension, the proper geometric dimension, and the virtual cohomological dimension of $\calN_g$ are equal whenever $g\neq 4,5$. In particular, there exists a model for the classifying space  of $\calN_g$ for proper actions of dimension $\vcd(\calN_g)=2g-5$. 
Similar results are obtained for the mapping class group of a non-orientable surface with boundaries and possibly punctures, and for the pure mapping class group of a non-orientable surface with punctures and without boundaries.
\end{abstract}
\maketitle

\section{Introduction}

Let $G$ be a group. In the literature we can find several notions of dimension defined for $G$. In the present paper we are mainly interested in the \emph{proper geometric dimension} $\gdfin(G)$, the \emph{proper cohomological dimension} $\cdfin(G)$, and the virtual cohomological dimension $\vcd(G)$ of $G$ (provided $G$ is virtually torsion free).

Let $\calF$ be the family of finite subgroups of $G$. A model for the classifying space of $G$ for proper actions $\underline{E}G$ is a $G$-CW-complex $X$ such that the fixed point set $X^F$ is empty if $F\notin \calF$ and contractible otherwise, in particular, $X^F$ is non-empty if $F$ is a finite subgroup of $G$. Such a model always exists and is unique up to proper $G$-homotopy. The \emph{proper geometric dimension of $G$}, denoted $\gdfin(G)$, is the minimum $n$  for which there exists an $n$-dimensional model for $\underline{E}G$.

On the other hand, we have the so-called restricted orbit category $\calO_\calF G$, which has as objects the homogeneous $G$-spaces $G/H$, $H\in \calF$, and morphisms are  $G$-maps. A  \emph{$\calO_\calF G$-module} is a contravariant functor from $\calO_\calF G$ to the category of abelian groups, and a morphism between two $\calO_\calF G$-modules is a natural transformation of the underlying functors. Denote by $\calO_\calF G$-mod the category of $\calO_\calF G$-modules. It turns out that $\calO_\calF G$-mod is an abelian category with enough projectives. Thus we can define a $G$-cohomology theory for $G$-spaces $H_\calF^*(-;M)$ for every $\calO_\calF G$-module $M$ (see \cite[p.~7]{MV03}).  The
\emph{proper cohomological dimension of $G$}---denoted $\cdfin(G)$---is the
largest non-negative $n\in\dbZ$ for which the cohomology group $H^n_{\calF}(G;M)=H^n_{\calF}(E_\calF G;M)$
is nontrivial for some $M\in\mathcal{O}_{\calF}G \textrm{-mod}$. Equivalently, $\cd_\calF(G)$ is the length of the shortest projective resolution of the constant $\orf{G}$-module $\dbZ_\calF$, where $\dbZ_\calF$ is given by $\dbZ_\calF(G/H)=\dbZ$, for all $H\in \calF$, and every morphism of $\orf{G}$ goes to the identity function.

The \emph{cohomological dimension $\cd(H)$ of a group $H$} is the length of shortest projective resolution, in the category of $H$-modules, for the trivial $H$-module $\dbZ$.
Provided $G$ is virtually torsion free, this is, that $G$ contains a torsion free subgroup $H$ of finite index, the \emph{virtual cohomological dimension of $G$} is defined to be $\vcd(G)=\cd(H)$. A well-known theorem of Serre stablishes that $\vcd(G)$ is well definied, that is, it does not depend on the choice of the finite index torsion free subgroup $H$ of $G$ (see for example \cite[p.~190]{Br94}).\\

For every group $G$, by \cite[Theorem~2]{BLN01} we have the following inequalities
\begin{equation}\label{eq:inequalities:dimensions}
    \vcd(G)\leq \cdfin(G)\leq \gdfin(G)\leq \max\{3,\cdfin(G)\}.
\end{equation}
    
As a consequence, if $\cdfin(G)\geq 3$, then $\cdfin(G)=\gdfin(G)$. Moreover, the only possible scenario where  $\cdfin(G)$ might be non-equal to $\gdfin(G)$, is provided by the existence of a group $G$ with $\cdfin(G)=2$  and $\gdfin(G)=3$. The existence of such a group is unkwown.  The first inequality in \eqref{eq:inequalities:dimensions} may be strict as proved in \cite{LN03,LP17,DS17}. Also the second inequality may be strict, this is known as the (generalized) Eilenberg-Ganea problem. Examples of groups for which the second inequality is strict are constructed in \cite{BLN01,LP17,SS19}.

On the other hand, the $\vcd(G)$ is known to be equal to $\cdfin(G)$ for the following classes of groups: elementary amenable groups of type $\mathrm{FP}_\infty$ \cite{KMPN09}, $\mathrm{SL}_n(\dbZ)$ \cite{Ash84}, $\mathrm{Out}(F_n)$ \cite{Vo02}, the mapping class group of any \emph{orientable} surface with boundary components and punctures \cite{AMP14}, any lattice in a classical simple Lie group \cite{ADMS17}, any lattice in the group of isometries of a symmetric space of
non-compact type without Euclidean factors \cite{La19}, groups acting chamber transitively on a Euclidean building \cite{DMP16}, and groups satisfying properties (M), (NM) and that admit a cocompact model for $\underline{E}G$ \cite{SS20}.\\

Let $\Ngnb$ denote the connected  \emph{non-orientable} surface of genus $g$ with $n$ distinguished points (also called punctures) and $b$ boundary components. The \emph{mapping class group} $\mcg(\Ngnb) = \calNgnb$ of $\Ngnb$ is the group of isotopy classes of self-homeomorphisms of $\Ngnb$ which take the set of distinguished points to itself and fix the boundary components pointwise. For an orientable surface $\Sgnb$ of genus $g$ with $n$ distinguished points and $b$ boundary components, the \emph{mapping class group} $\mcg(\Sgnb)=\Gammagnb$ is defined similarly but now considering only orientation preserving homeomorphisms. For $n\geq 1$, the \emph{pure mapping class group of $\Ngnb$}, denoted $P\calNgnb$, is the subgroup of $\calNgnb$ of elements that fix pointwise the set of punctures.

Whenever we consider a a surface without punctures or boundaries, we will omit the corresponding index from the notation. For instance, $\Ngn{g}{n}$ is the surface with $n$ punctures and genus $g$ without boundary components, and $\calNgn{g}{n}$ is its corresponding mapping class group.

In \cite{AMP14} Aramayona and Martínez-Pérez proved that, for all $g\geq 0$,
\[\cdfin(\Gamma_g)=\gdfin(\Gamma_g)=\vcd(\Gamma_g).\] Moreover, since $\Sgnb$ is torsion free for $b>0$, Aramayona and Martínez-Pérez obtained as a corollary of their theorem that, for all $b,g,n\geq 0$,
\[\cdfin(\Gammagnb)=\gdfin(\Gammagnb)=\vcd(\Gammagnb).\]

In the present paper we obtain the analogue of Aramayona and Martínez-Perez theorem for the non-orientable case. We closely follow their strategy. Our main result is the following.

\begin{theorem}\label{thm:main}
Let $g\geq 1$. If $g\neq4,5$, then
\[ \vcd(\calN_g)=\cdfin(\calN_g)=\gdfin(\calN_g). \]
\end{theorem}

The first natural thing to notice out of our main theorem is that we are excluding the case $g=4,5$. We do not know whether the conclusion of the main theorem holds in these cases. Anyway, we obtained the following theorem.

\begin{theorem}\label{thm:exeptional:cases}
\[3=\vcd(\calN_4)\leq \cdfin(\calN_4)\leq \vcd(\calN_4)+3=6 \]

and

\[5=\vcd(\calN_5)\leq \cdfin(\calN_5)\leq \vcd(\calN_5)+1=6.\]
\end{theorem}

The following corollary is a straightforward consequence of \cref{thm:main} and the definition of $\gdfin(\calN_g)$, but we include it for completeness.

\begin{corollary}\label{corollary:main}
There exists a model for $\underline{E}\calN_g$ of dimension
\[\vcd(\calN_g)=\begin{cases}
0 & \text{if }g= 1,2,\\
2g-5 & \text{if }g=3 \text{ or }g\geq 6.
\end{cases}\]
Moreover, this is the minimal dimension possible for a model of $\underline{E}\calN_g$.
\end{corollary}

As a remark, we do not know if the model in the statement of \cref{corollary:main} can be realized as a subspace of the Teichmüller space described in \cite{PP16}, i.e. we do not know whether the Teichmüller space of $\calN_g$ has a spine of dimension $\vcd(\calN_g)$.

In the case we have a surface $N_{g,n}$ with $n\geq 1$ we obtained the following result as a consequence of \cref{thm:main} and \cref{thm:exeptional:cases}.

\begin{theorem}\label{thm:main:with:punctures}
Let $n\geq 1$. 
\begin{enumerate}
 \item If $g=1,2,3 $ or $g\geq 6$, then $\gdfin (P\calN_{g,n})=\cdfin (P\calN_{g,n}) =\vcd (P\calN_{g,n})$.
  \item If $g=4$, then 
 $\vcd (P\calN_{g,n})\leq \gdfin (P\calN_{g,n}) \leq \vcd (P\calN_{g,n}) +3$.
 \item If $g=5$, then 
 $\vcd (P\calN_{g,n})\leq \gdfin (P\calN_{g,n}) \leq \vcd (P\calN_{g,n}) +1$.
\end{enumerate}
\end{theorem}

Note that the previous theorem  deals with the pure mapping class group of $\Ngnb$ rather than the full mapping class group. See \cref{section:questions} for more details.

Finally, in the case we have at least one boundary component we have the following result.

\begin{theorem}\label{thm:vcd:gd:boundary}
If $b\geq 1$, then $\vcd (\calN_{g,n}^b)=\cdfin (\calN_{g,n}^b)= \gdfin (\calN_{g,n}^b)$.
\end{theorem}

The present paper is organized as follows. In \cref{section:preliminaries} we state \cref{thm:aramayona:martinezperez}, which is the criterion we will use in order to prove \cref{thm:main}. Also in \cref{section:preliminaries} we set up our main technical tools such as the Nielsen realization theorem for non-orientable surfaces, explicit computations of $\vcd(\calNgnb)$ and $\vcd(\Gammagnb)$, the Riemann-Hurwitz formula, and a computation of the $\vcd$ of the Weyl group of a finite subgroup $F$ of $\calN_g$ in terms of the $\vcd$ of certain mapping class groups. In \cref{section:sueful:results} we state and prove several inequalities that will be crucial in the proof of the main theorem. In Sections \ref{section:non-orientable:case} and \ref{section:orientable:case}, we verify the hypothesis of \cref{thm:aramayona:martinezperez} when the orbifold $N_g/F$ is non-orientable and orientable respectively. In \cref{section:proof:main:thm} we prove all the results stated in this introduction. Finally, in \cref{section:questions} we state some questions that arise naturally from the statements of our theorems.


\subsection*{Acknowledgements} The authors would like to thank Jesús Hernández Hernández for several helpful conversations. C.E.H. received support from PAPIIT-IN105318 and thanks to CCM-UNAM for the facilities provided in the preparation of this article. L.J.S.S was supported by PAPIIT-IA101221.


\section{Preliminaries}\label{section:preliminaries}

\subsection{Aramayona and Martínez-Pérez criterion}
Let $F$ be a finite subgroup of $G$. We denote $N_G(F)$, and $W_G(F)=N_G(F)/F$  the normalizer, and the Weyl group of $F$ respectively. If there is no risk of confusion we will omit the parenthesis and subindices, i.e. we will use the notation $NF$, and $WF$.

The \emph{length} $\lambda(F)$ of a finite group $F$ is the largest $i\geq 0$ for which  there is a sequence \[1=F_0<F_1 < \cdots < F_i=F.\]

The following theorem is a mild generalization of \cite[Theorem~3.3]{AMP14}, and the proof is exactly the same as in the Aramayona and Martínez-Pérez reference.

\begin{theorem}\label{thm:aramayona:martinezperez} Let $G$ be a virtually torsion free group and let $m\geq 0$. Assume that for any $F\leq G$ finite, $\vcd(W_G(F))+\lambda(F)\leq m$. Then $\cdfin(G)\leq m$. In particular, if $m=\vcd(G)$, then $\cdfin(G)=\vcd(G)$.
\end{theorem}

The proof of our main theorem will be based on verifying the hypothesis of \cref{thm:aramayona:martinezperez} when $G=\calN_g$.

\subsection{Nielsen realization theorem}

    An important result in the study of mapping class groups is the Nielsen realization theorem, it responds affirmative to the question of whether a finite group of the mapping class grouop  of a surface arises as a group of isometries of some hyperbolic structure. In the literature, the theorem is usually  enunciated for orientable surfaces (see \cite{KS83} and  \cite[Theorem~7.2]{FM12}  ) but, of course is also valid for non-orientable ones (see  \cite[Remark~on~p.~256]{KS83}). In this work we need the version for non-orientable surfaces  and  for the sake of completeness we state it here. Denote the Euler characteristic of the surface $N$ by $\chi(N)$.
    
    \begin{theorem}\label{thm:Nielsen:realization}
    Let $N = N_{g,n}$ non-orientable and suppose $\chi(N) < 0$. Suppose $F \leq \calNgn{g}{n}$ is a finite group. Then there exists a finite group $\widetilde{F}\leq Homeo (N)$ so that the natural projection
    $$Homeo(N) \longrightarrow \calNgn{g}{n} $$
    restricts to an ismorphism $$\widetilde{F} \longrightarrow F.$$
    Further, $\widetilde{F}$ can be chosen to be a subgroup of isometries of some hyperbolic metric of $N$.
    \end{theorem}

\subsection{Virtual cohomological dimension of mapping class groups} 
In \cite{Ha86} Harer computed the virtual cohomological dimension of an orientable surface with punctures and boundary components. In \cite[Theorem~6.9]{I87} Ivanov computed the virtual cohomological dimension of the mapping class group of a non-orientable surface $N_g$ of genus $g$ and $n$ marked points. We borrowed the following formulas from Ivanov's paper 

\[
\vcd(\Gammagn) =
\begin{cases}
0 & \text{if } g=0 \text{ and } n\leq 3 \\
n-3 & \text{if } g=0 \text{ and } n\geq 3 \\
1 & \text{if } g=1 \text{ and } n= 0 \\
n & \text{if } g=1 \text{ and } n\geq 1 \\
4g-5 & \text{if } g\geq 2 \text{ and } n=0\\
4g+n-4 & \text{if } g\geq 2 \text{ and } n\geq 1
\end{cases}
\]

\[
\vcd(\calNgn{g}{n}) =
\begin{cases}
0 & \text{if } g=1 \text{ and } n\leq 2 \\
n-2 & \text{if } g=1 \text{ and } n\geq 3 \\
n & \text{if } g=2 \\
2g-5 & \text{if } g\geq 3 \text{ and } n=0\\
2g+n-4 & \text{if } g\geq 3 \text{ and } n\geq 1
\end{cases}
\]

For our purposes we will need a concrete formula for the virtual cohomological dimension of a non-orientable surface with boundary components. Since we lack of a reference for this, we will compute it in \cref{prop:vcd:mpcg}. Before proving \cref{prop:vcd:mpcg} we need the following lemma.

\begin{lemma}\label{thm:main:with:boundaries}
Let $b\geq 1$. Then the group $\calNgnb$ is torsion-free.  
\end{lemma}
\begin{proof}
First assume  $\chi(N_{g,n}^b) < 0$, or equivalently $g-2+n+b \geq 1$. Let $S = S_{g-1,2n}^{2b}$ 
be the oriented double cover of $N_{g,n}^{b}$ and  $\tau:S \rightarrow S$ be the covering involution.

Let $[f]\in \calNgnb$  such that $[f]^{m}=[id]$. The homeomorphism $f$ can be lifted to a  orientation preserving homeomorphism $\widetilde{f}$ of $S$. Notice that $\widetilde{f}$ fix the boundary components of $S$ pointwise, even more $\widetilde{f}^{m}\simeq id$. As $\Gamma_{g-1,2n}^{2b}$ is torsion free (see \cite[Corollary~7.3]{FM12}) we have that $[\widetilde{f}] = [id]$. As $\widetilde{f}$ is $\tau$-equivariant  and isotopic to the identity map, then it is $\tau$-equivariant isotopic to the identity map (see \cite{Zi73}), in consequence $f\simeq id$ in $N_{g,n}^{b}$.

If $\chi(N_{g,n}^b)\geq 0$, as $g\geq 1$ and $b\geq 1$, we necessarily have $g=1$, $b=1$, $n=0$, that is $N_{g,n}^b$ is  the Möbius band, but in this case the mapping class group is trivial (see \cite{EP1966}).

\end{proof}

\begin{proposition}\label{prop:vcd:mpcg} 
\[
\vcd(\Gammagnb) =
\begin{cases}
b & \text{if } g=0 \text{ and } n+b\leq 3 \\
n+2b-3 & \text{if } g=0 \text{ and } n+b\geq 3 \\
1+b & \text{if } g=1 \text{ and } n+b= 0 \\
n+2b & \text{if } g=1 \text{ and } n+b\geq 1 \\
4g-5 & \text{if } g\geq 2 \text{ and } n+b=0\\
4g+n+2b-4 & \text{if } g\geq 2 \text{ and } n+b\geq 1
\end{cases}
\]

\[
\vcd(\calNgn[b]{g}{n}) =
\begin{cases}
b & \text{if } g=1 \text{ and } n+b\leq 2 \\
n+2b-2 & \text{if } g=1 \text{ and } n+b\geq 3 \\
n+2b & \text{if } g=2 \\
2g-5 & \text{if } g\geq 3 \text{ and } n+b=0\\
2g+n+2b-4 & \text{if } g\geq 3 \text{ and } n+b\geq 1
\end{cases}
\]
\end{proposition}
\begin{proof}  Suppose $b>0$.  
The following short exact sequence, which is analogous to that in \cite[Proposition~3.19]{FM12}, appears in \cite[p.~262]{St10}
\begin{equation}\label{sec-mcg}
    1\to \dbZ^b\to P^k\calNgnb \to P^{k+b}\calNgn{g}{n+b} \to 1
\end{equation}
where the second and third term of the sequence  are  finite index subgroups of the pure mapping class groups $P\calNgnb$  and $P\calNgn{g}{n+b}$, respectively.

On the other hand by \cite[Theorem~6.9]{I87} we know that  $\calNgn{g}{n+b}$ is a virtual duality group in the sense of Bieri and Eckman. Therefore $P^{k+b}\calNgn{g}{n+b}$ and $\dbZ^b$ are both virtual duality groups of dimension  $\vcd(\calNgn{g}{n+b})$ and $b$ respectively. Note that $\calNgnb$ is torsion-free, by \cref{thm:main:with:boundaries}, and $\calNgn{g}{n+b}$ is virtually-torsion free.  We can choose a finite index torsion free duality subgroup $G$ of $ P^{k+b} \calNgn{g}{n+b}$ and from (\ref{sec-mcg}) we get a short exact sequence $1\to \mathbb{Z}^b \to H \to G\to 1$, $H$ is finite index subgroup of $P
^{k}\calNgnb$. By \cite[(ii)~on~page~88]{I87}, $H$ is a duality group of dimension $b+\vcd(\calNgn{g}{n+b})$. Therefore $\calNgnb$ is a virtual duality group of dimension $b+\vcd(\calNgn{g}{n+b})$. This implies that $\vcd(\calNgnb)=b+\vcd(\calNgn{g}{n+b})$. Now the result follows from Ivanov's computations in the empty-boundary case.

The proof for the orientable case is analogous (see also \cite[Theorem~4.1]{Ha86}).  
\end{proof}

\subsection{$2$-dimensional orbifolds and the Riemann-Hurwitz formula}
All the content in this section is standard material and can be found in \cite{Th97}.

Recall that the singular locus  of a 2-dimensional (closed) orbifold has one of the following three local models:

\begin{enumerate}
    \item The mirror: $\dbR^2/(\dbZ/2)$, where $\dbZ/2$ acts by reflection on one of the axis.
    \item Elliptic points of order $n$: $\dbR^2/(\dbZ/n)$, where $\dbZ/n$ acts by rotations.
    \item Corner reflectors  of order $n$: $\dbR^2/D_n$, where $D_n$ is the dihedral group of order $2n$, that  is generated by reflections about two lines that meet at an angle of $\pi/n$.
\end{enumerate}

Let $O$ be a 2-dimensional orbifold with underlying topological space $X_O$ with $k$ corner reflectors of orders $p_1,\dots, p_k$ and $l$ elliptic points of orders $q_1,\dots , q_l$.  The Riemann-Hurwitz formula for the (orbifold) Euler characteristic is the following
\[
\chi(O)=\chi(X_O)-\frac{1}{2}\sum_{i=1}^{k}\left(1-\frac{1}{p_i}\right)-\sum_{i=1}^{l}\left(1-\frac{1}{q_i}\right).
\]

All of the orbifolds appearing in this paper arise in the following form. Let $F$ be a finite group acting on (a possibly non-orientable surface)  $S=\Sgnb$, then the quotient space $O_F$ is an orbifold with underlying topological space $S_F$. Since $F$ may have elements acting as non-orientation preserving homeomorphisms, $S_F$ may be orientable or non-orientable (see for instance \cite[Corollary~3.2]{GS16} and \cite{Conder15}), and may have elliptic points, mirror points and corner reflectors. In this situation we will use the following notation:

\begin{itemize}
    \item $g_F$ is the genus of $S_F$. 
    \item $e_F$ is the number of elliptic points of $O_F$ and the orders will be denoted $q_1,\dots ,q_{e_F}$.
    \item $c_F$ is the number of corner points of $O_F$ and the orders will be denoted $p_1,\dots ,p_{c_F}$.
    \item $b_m$ is the number of boundary components of $S_F$ that do not contain any corner point. In other words, all points in such boundary components are mirror points.
    \item $b_c$ is the number of boundary components of $S_F$ that contain at least one corner point.
    \item $b=b_m+b_c$ is the number boundary components of $S_F$.
    \item $E_F=\sum_{i=1}^{e_F}\left(1-\frac{1}{q_i}\right)$ and $C_F=\sum_{i=1}^{c_F}\left(1-\frac{1}{p_i}\right)$. So that the we can rewrite the Euler characteristic of $O_F$ as $\chi(O_F)=\chi(S_F)-C_F/2 - E_F$.
    \item If there is no risk of confusion we will not distinguish between an orbifold and its underlying topological space.
    
\end{itemize}

The orbifold Euler characteristic is defined in such a way that satisfies the following multiplicativity property, which is the so-called Riemann-Hurwitz formula:

\begin{equation}\label{eq:riemann:hurwitz}
|F|\chi(O_F)=\chi(S).    
\end{equation}

The following inequalities are clear and they will be useful latter:

\begin{equation}\label{eq:inequalities:eF:cF}
    \frac{e_F}{2}\leq E_F\leq e_F\quad \text{ and }\quad \frac{c_F}{2}\leq C_F\leq c_F.
\end{equation}


\subsection{Weyl groups of finite groups in the non-orientable mapping class groups}  In this section we prove \cref{lemma:maher}, that is a mild generalization of \cite[Proposition~2.3]{Ma11}. The proof is also an adaptation of Maher's argument, still we include it here for the sake of completeness. Later on we prove \cref{thm:vcd:weyl:groups}, this result provides a way to compute the $\vcd$ of Weyl groups of finite subgroups of $\calN_g$ in terms of the $\vcd$ of other mapping class groups. \cref{thm:vcd:weyl:groups} is one of the main tools in the proof of \cref{thm:main} and \cref{thm:exeptional:cases}.

Let $F$ be a finite subgroup of the mapping class group $\calN_g$ of $N_g$ with $g\geq 3$. By  \Cref{thm:Nielsen:realization}  there exists a hyperbolic metric on $N$ such that $F$ is isormorphic to a finite group of isometries with respect to that metric. Denote by $O_F$ the quotient orbifold. Define $\Gamma^*_F$ as the group of isotopy classes of self-homeomorphisms of $O_F$ that send elliptic points of order $q$ to elliptic points of order $q$, mirror points to mirror points and corner reflectors of order $p$ to corner reflectors of order $p$.

\begin{lemma}\label{lemma:maher}
Let $WF$ be the Weyl group of $F$ in $\calN_g$. Then, there exists an injective homomorphism $WF\to \Gamma^*_F$ such that the image is a finite index subgroup of $\Gamma^*_F$.
\end{lemma}
\begin{proof}
Let $K$ be the subgroup of $\calN_g$ of elements that admit a representative $N_g\to N_g$ that preserves the fibers of the projection $N_g\to O_F$. In particular $F\leq K$. By \cite[last~paragraph~in~p.~20]{Zi73}, every two isotopic fiber preserving self-homeomorphisms of $N$ are isotopic via a fiber preserving isotopy. Hence we have a well-defined map $\varphi\colon K\to \Gamma^*_F$ with kernel $F$.

On the other hand, by definition of  $\Gamma^*_F$ all of its elements are represented by an orbifold map $O_F\to O_F$. Hence we have an action of $\Gamma^*_F$ on the orbifold fundamental group $\pi_1^{\mathrm{orb}}(O_F)$ by automorphisms. For more details on the orbifold fundamental group see \cite[III.$\mathcal{G}$.3]{BH99}.

Let us characterize the elements in $\varphi(K)$. A homeomorphism $g\colon O_F\to O_F$ is covered by a map $\tilde g\colon N_g\to N_g$ if and only if the induced map  $g_*\colon \pi_1^{\mathrm{orb}}(O_F)\to \pi_1^{\mathrm{orb}}(O_F)$ restricts to an automorphism of $\ker(\theta)$, where $\theta\colon \pi_1^{\mathrm{orb}}(O_F)\to F$ is the map given by the \emph{normal} covering $N\to O_F$. Therefore the elements of $\varphi(K)$ are those elements of $\Gamma^*_F$ represented by maps $g:O_F \to O_F$ such that $g_*$ leave $\ker(\theta)$ invariant.

Note that $\ker(\theta)$ is a finite index subgroup of $\pi_1^{\mathrm{orb}}(O_F)$, say $l$. Since $\pi_1^{\mathrm{orb}}(O_F)$ is a finitely presented group, it has finitely many subgroups of  index $l$. Let $X$ be the finite set of subgroups of $\pi_1^{\mathrm{orb}}(O_F)$ of index $l$. Hence $\Gamma^*_F$ acts on $X$ and the kernel of this action is contained in $\varphi(K)$. Since this kernel is clearly a finite index subgroup of $\Gamma^*_F$, we conclude that $\varphi(K)$ is a finite index subgroup of $\Gamma^*_F$.

Once we prove that $K$ is equal to the normalizer of $F$ in $\calN_g$, we will finish our proof. This was proved by Zieschang in Corollary~8.7 and the last paragraph on page 20 of \cite{Zi73}.

\end{proof}


Before setting the next result notice that in the definition on mapping class group with punctures, we can think of the punctures as boundaries that are not fixed pointwise by homeomorphisms and  where  homotopies can move the points of these boundaries.  With this in mind we  think of the underlying topological surface $S_{F}$  of the orbifold  $O_{F}$ as a surface with $b_{c}$ boundary components, the boundaries of $S_{F}$ that have at least one corner point, and $e_{F} + b_{m}$ punctures where $e_{F}$ punctures come from the elliptic points of $O_{F}$ and the rest $b_{m}$ punctures are the boundary components that do not contain any corner point.

\begin{theorem}\label{thm:vcd:weyl:groups}
Let $F$ be a finite subgroup of $\calN_g$, then the Weyl group $WF$ of $F$, is commensurable with $\calNgn[b_c]{g_F}{b_m+e_F}$ if $O_F$ is non-orientable, or with $\calSgn[b_c]{g_F}{b_m+e_F}$ if $O_F$ is orientable.

In particular \[\vcd(WF)=\begin{cases}
\vcd(\calNgn[b_c]{g_F}{b_m+e_F}) & \text{if $O_F$ is a non-orientable surface,} \\
\vcd(\calSgn[b_c]{g_F}{b_m+e_F}) & \text{if $O_F$ is an orientable surface.}
\end{cases}\] 
\end{theorem}

\begin{proof}
We are going to prove that $\Gamma^*_F$ is commensurable with $\calNgn[b_c]{g_F}{b_m + e_F}$ if $O_F$ is non-orientable or  with $\calSgn[b_c]{g_F}{b_m + e_F}$ if $O_F$ is orientable.

We think of the underlying surface $S_F$ of $O_F$ as a surface with $b_c$ boundaries and $b_m + e_F$ punctures (see the paragraph before the statement of the theorem). Recall that $P\calNgn[b_c]{g_F}{b_m + e_F} $ (respectively $P\calSgn[b_c]{g_F}{b_m + e_F}$) denotes the finite index subgroup of $\calNgn[b_c]{g_F}{b_m + e_F}$ (respectively $\calSgn[b_c]{g_F}{b_m + e_F}$) consisting of homotopy classes of those homeomorphisms which fix the set of punctures pointwise. Similarly denote by $P\Gamma^*_F$ the subgroup of $\Gamma^*_F$ consisting of those elements which fix each elliptic and corner point and leaves each boundary component of the underlying surface $S_F$ invariant as a set. Notice that in this case we also have that $P\Gamma^*_F$ is a finite index subgroup. We claim that if $O_F$ is non-orientable, the groups $P\calNgn[b_c]{g_F}{b_m + e_F}$ and $P\Gamma^*_F$ are conmensurable.

Consider the map 
\begin{align*}
\theta : P\calNgn[b_c]{g_F}{b_m + e_F} &\longrightarrow  P\Gamma^*_F\\
[f]      &\longmapsto      [f]
\end{align*}

The map $\theta$ is well defined because for every $[f]\in P\calNgn[b_c]{g_F}{b_m + e_F}$, the homeomorphism $f$ fixes each puncture  and each boundary component of $S_F$ pointwise, in particular it fixes those points that correspond to corner points. Even more, if two of those homeomorphisms are in the same class, the homotopy between them fixes each puncture and each boundary point.
It is clear that $\theta$ is a group homomorphism and, even more, it is injective. 

We want to know who is the image of $\theta$. Let $[f]\in P\Gamma^*_F$, then $f$ restricts to a homeomorphism $f|_{b_i}$ on each of the $b_c$ boundaries with corner points. We define the following group homomorphism

\begin{align*}
D:P\Gamma^*_F & \longrightarrow  \Z /2 \times \cdots \times \Z /2 \\
  [g] & \longmapsto   \left( \mathrm{dg}(g|_{b_1}), \dots , \mathrm{dg}(g|_{b_{b_c}}) \right) 
\end{align*}

\noindent where $\mathrm{dg}(f)$ is the degree of a homeomorphism $f$ from the circle to the circle. Clearly $\Ima(\theta) \subset \Ker (D)$ and as every homeomorphism of the circle with $\mathrm{dg}(f)=1$ is homotopic to the identity, we have that every $[g]\in \Ker (D)$ has a representant $g'$ that leaves the boundaries $b_1,\dots,b_{b_c}$ fixed pointwise. Then $\Ima(\theta) = \Ker(D)$. As $\Ker(D)$ has finite index in $P\Gamma^*_F$, then $P\calNgn[b_c]{g_F}{b_c + e_F}$ is commensurable with $P\Gamma^*_F$ and therefore $\Gamma^*_F$ is commensurable with $\calNgn[b_c]{g_F}{b_m + e_F}$. From the previous lemma we have that

\[
\vcd(WF) = \vcd(\calNgn[b_c]{g_N}{b_m+e_F}).
\]

If $O_F$ is orientable a similar argument proves that

\[
\vcd(WF) = \vcd(\calSgn[b_c]{g_N}{b_m+e_F}).
\]
\end{proof}

\section{Some useful results} \label{section:sueful:results}

\begin{lemma}\label{lemma:lambda:inequalities}
Let $F$ be a finite group, then
\[\lambda(F)\leq \frac{|F|}{2}\quad\text{and}\quad\lambda(F)\leq \log_2(|F|).\]
\end{lemma}
\begin{proof}
For $F=1$ the result is trivially true. Since for $|F|=2,3$ the length $\lambda(F)=1$, the result is clearly true in these cases. For $|F|\geq4$, $\log_2(|F|)\leq\frac{|F|}{2}$. Hence it is enough to prove the second inequality. 

As a consequence of Langrange's theorem $\lambda(|F|)\leq p$ where $p$ is the number of prime factors in the prime decomposition of $|F|$. The number $p$ can be as large as possible exactly when all prime divisors of $|F|$ are equal to 2, that is $p\leq \log_2(|F|)$. Now the result follows.
\end{proof}

\begin{lemma}\label{lemma:ab:inequality}
Let $a$ and $b$ be natural numbers such that $a\geq 1$ and $b\geq 2$. Consider the following inequality
\[ a+\frac{b}{2}\leq b(a-\epsilon)-1 \]

\begin{enumerate}
\item\label{lemma:ab:inequality:zero} If $\epsilon = 0$, then the inequality holds except when $(a,b)=(1,2)$ or $(1,3)$.
\item\label{lemma:ab:inequality:onehalf} If $\epsilon = \frac{1}{2}$, then the inequality holds except when $(a,b)=(1,b)$ or $(2,2)$.
\item\label{lemma:ab:inequality:one} If $\epsilon = 1$, then the inequality holds except when $(a,b)=(1,b)$, $(2,2)$, $(2,3)$, $(2,4)$, $(2,5)$ or $(3,2)$.
\end{enumerate}
\end{lemma}
\begin{proof}
The inequality in our statement is equivalent to each of the following inequalities:

\begin{gather*}
2a+b\leq 2b(a-\epsilon)-2,\\
2a+b\leq 2ab-2b\epsilon-2\\
2\leq 2ab-2b\epsilon -2a-b
\end{gather*}

Let us proceed to prove each one of the claims.

\begin{enumerate}
    \item Let $\epsilon=0$. Assume $a\geq1$. Then 
    $a(2b-2)-b\geq (2b-2)-b=b-2$. Hence we want $b-2\geq 2$  which is true if and only if $b\geq 4$. We conclude our inequality is true if $a\geq1$ and $b\geq 4$.
    
    \noindent Now assume that $b\geq 2$. Then $b(2a-1)-2a\geq 2(2a-2)-2a=2a-4$. Hence we want $2a-4\geq2$ which is true if and only if $a\geq 3$. We conclude our inequality is true if $a\geq 3$ and $b\geq 2$. Therefore the inequality is true except possibly when $(a,b)=(1,2)$, $(1,3)$, $(2,2)$, and $(2,3)$. One can verify by hand that the inequality holds for $(2,2)$ and $(2,3)$.
    
    \item Let $\epsilon=\frac{1}{2}$. Clearly the inequality is not true when $a=1$. Assume $a\geq 2$, then, $a(2b-2)-2b\geq 2(2b-2)-2b=2b-4$. We want $2b-4\geq 2$ which is true if and only if $b\geq 3$. Therefore the inequality is true when $a\geq 2$ and $b\geq 3$.
    
    \noindent Now assume $b\geq 2$. Then $b(2a-2)-2a\geq 2(2a-2)-2a=2a-4$. We want $2a-4\geq 2$ which is trus if and only if $a\geq 3$. Therefore the inequality is true when $a\geq 3$ and $b\geq 2$. We conclude the inequality is true with the exception of the pairs $(a,b)=(1,b)$, and $(2,2)$.
    
    \item This part can be proved under the same lines as the previous ones. 
\end{enumerate}

\end{proof}

\begin{proposition}\label{prop:vcd:inequalities}
Let $F$ be a finite subgroup of $\calN_g$ and assume that there is an $\epsilon$ such that
\[
\vcdNgoverF\geq {\vcd(WF)}-\epsilon.
\]
\
Then:
\begin{enumerate}
    \item\label{prop:vcd:inequalities:zero} If $\epsilon=0$, then $\vcd(WF)+\lambda(F)\leq \vcd(\calN_g)$ when $g\geq 4$ and $\vcd(WF)\geq 1 $.
    \item\label{prop:vcd:inequalities:onehalf} If $\epsilon=\frac{1}{2}$, then $\vcd(WF)+\lambda(F)\leq \vcd(\calN_g)$ when $g\geq 4$ and $\vcd(WF)\geq 2$.
    \item\label{prop:vcd:inequalities:one} If $\epsilon=1$, then $\vcd(WF)+\lambda(F)\leq \vcd(\calN_g)$ when $g\geq 5$ and $\vcd(WF)\geq 2$. Moreover, if $g=4$, then $\vcd(WF)+\lambda(F)\leq \vcd(N_g)+1=4$ provided $\vcd(WF)\geq 2$.
\end{enumerate}
\end{proposition}
\begin{proof}
The inequality of the statement is equivalent to $|F|(\vcd(WF)-\epsilon)-1\leq \vcd(\calN_g)$. If we want to prove $\vcd(WF)+\lambda(F)\leq \vcd(\calN_g)$, it is enough to prove $\vcd(WF)+\lambda(F)\leq |F|(\vcd(WF)-\epsilon)-1$. On the other hand by \Cref{lemma:lambda:inequalities}, it will suffice to prove
\begin{equation}\label{eq:lemma:vcd:inequalities}
\vcd(WF)+\frac{|F|}{2} \leq |F|(\vcd(WF)-\epsilon)-1.
\end{equation}

Lets proceed to prove each item in our statement.

\begin{enumerate}
    \item Let $\epsilon=0$. Then by \cref{lemma:ab:inequality} (\ref{lemma:ab:inequality:zero}), equation \eqref{eq:lemma:vcd:inequalities} is true except when $(\vcd(F),|F|)=(1,2)$ or $(1,3)$. Now, we only have to prove our claim for these exceptional cases. For $|F|=2$ or $3$, $\lambda(|F|)=1$. Hence, for the exceptional cases we have $\vcd(WF)+\lambda(F)\leq 2$. By \cref{prop:vcd:mpcg}, $\vcd(\calN_g)=2g-5$ for $g\geq 3$. Therefore $\vcd(WF)+\lambda(F)\leq \vcd(\calN_g)$ for $g\geq 4$.

    \item Let $\epsilon=\frac{1}{2}$. Then by \cref{lemma:ab:inequality} (\ref{lemma:ab:inequality:onehalf}), equation \eqref{eq:lemma:vcd:inequalities} is true except when $(\vcd(F),|F|)=(1,|F|)$ or $(2,2)$. For $|F|=2$, $\lambda(|F|)=1$. Hence, for the exceptional case we have $\vcd(WF)+\lambda(F)\leq 3$. By \cref{prop:vcd:mpcg}, $\vcd(\calN_g)=2g-5$ for $g\geq 3$. Therefore $\vcd(WF)+\lambda(F)\leq \vcd(\calN_g)$ for $g\geq 4$.

    \item Let $\epsilon=1$. Then by \cref{lemma:ab:inequality} (\ref{lemma:ab:inequality:one}), equation \eqref{eq:lemma:vcd:inequalities} is true except when $(\vcd(F),|F|)=(1,|F|)$, $(2,2)$, $(2,3)$, $(2,4)$, $(2,5)$,  or $(3,2)$. For $|F|=2,3,5$, $\lambda(|F|)=1$ and for $|F|=4$ we get $\lambda(|F|)=1$. Hence, for the exceptional case we have $\vcd(WF)+\lambda(F)\leq 4$. By \cref{prop:vcd:mpcg}, $\vcd(\calN_g)=2g-5$ for $g\geq 3$. Therefore $\vcd(WF)+\lambda(F)\leq \vcd(\calN_g)$ for $g\geq 4$. The moreover part also follows now.

\end{enumerate}

\end{proof}

\begin{proposition}\label{prop:vcd:geq:hurwitz}
Let $g\geq 3$, and let $F$ be a finite subgroup of $\calN_g$. Then
\[
\vcdNgoverF\geq \alpha g_F-4+e_F+b_m+2b_c+b_m+\frac{c_F}{2}
\]
where $\alpha=2$ if $O_F$ is non-orientable or $\alpha=4$ if $O_F$ is orientable.
\end{proposition}
\begin{proof}
The claim follows from the following chain of equalities and inequalities.
\begin{align*}
    \vcdNgoverF &= \frac{2g-4}{|F|}
    =\frac{-2\chi(N_g)}{|F|}
    =-2\chi(O_F)\\
    &=\alpha g_F+2b-4+C_F+2E_F\\
    &\geq \alpha g_F+2b+-4+\frac{c_F}{2}+e_F
\end{align*}
where the first equality is by \cref{prop:vcd:mpcg}, the third equality is by the Riemann-Hurwitz formula \eqref{eq:riemann:hurwitz}, and the last inequality follows from \eqref{eq:inequalities:eF:cF}.
\end{proof}

Now we state our last auxiliary result.

\begin{lemma}\label{lemma:smallest:possibilities}
Let $2\leq q\leq r\leq s$  be natural numbers. Then the smallest possible positive value for
\[k=1-\frac{1}{q}-\frac{1}{r}-\frac{1}{s}\]
is $\frac{1}{42}$ and it is reached exactly when $(q,r,s)=(2,3,7)$. If we additionally assume that two of the numbers $q$ , $r$, $s$ are equal, then the smallest possible value for $k$ is $\frac{1}{12}$ and it is reached when $(q,r,s)=(3,3,4)$.
\end{lemma}
\begin{proof}
It is clear that for $(q,r,s)=(2,3,7)$, $k=\frac{1}{42}$.
By hypothesis \[\frac{1}{q}\geq\frac{1}{r}\geq \frac{1}{s}.\]
If $4\leq q\leq r\leq s$, then $k\geq\frac{1}{4}$. Thus $q$ can only be equal to 2 or 3.

If $q=2$, then $\frac{1}{r}+\frac{1}{s}<\frac{1}{2}$. Thus $3\leq r\leq s$. If $5\leq r \leq s$, then $k\geq \frac{1}{10}>\frac{1}{42}$, thus we only have to consider $r$ to be equal 3 or 4. In the case $r=3$ we have $\frac{1}{s}<\frac{1}{2}-\frac{1}{3}=\frac{1}{6}$, hence the smallest possible value or $s$ is 7 and we get the triple $(q,r,s)=(2,3,7)$. In the case $r=4$ we have $\frac{1}{s}<\frac{1}{2}-\frac{1}{4}=\frac{1}{4}$, hence the smallest possible value for $s$ is 5 we get the triple $(q,r,s)=(2,3,5)$.

If $q=3$, then $\frac{1}{r}+\frac{1}{s}<\frac{2}{3}$. Thus $3\leq r\leq s$. If $5\leq r \leq s$, then $k\geq \frac{4}{15}>\frac{1}{42}$, thus we only have to consider $r$ to be equal 3 or 4. In the case $r=3$ we have $\frac{1}{s}<\frac{2}{3}-\frac{1}{3}=\frac{1}{3}$, hence the smallest possible value for $s$ is 4 we get the triple $(q,r,s)=(3,3,4)$. In the case $r=4$ we have $\frac{1}{s}<\frac{2}{3}-\frac{1}{4}=\frac{5}{12}$, hence the smallest possible value for $s$ is 4 we get the triple $(q,r,s)=(2,4,4)$.

Summarizing we get as possible candidates for the triple $(q,r,s)$ the following triples $(2,3,7)$, $(2,4,5)$, $(3,3,4)$, and $(3,4,4)$. Now the conclusions follow from a straightforward computation.
\end{proof}

Now we are ready to verify the hypothesis of \cref{thm:aramayona:martinezperez}. For this we will distinguish two cases: when $O_F$ is non-orientable and when $O_F$ is orientable. This cases are natural to be considered in separate arguments due to the difference in nature of $\vcd(WF)$ depending on the orientability of the associated orbifold $O_F$ as proved in \cref{thm:vcd:weyl:groups}.

\section{Verifying the hypothesis of \cref{thm:aramayona:martinezperez} when $O_F$ is non-orientable }\label{section:non-orientable:case}

Let $F$ be a finite group of $\calN_g$ such that the orbifold  $O_F=N_g/F$ is non-orientable. As a consequence of \cref{thm:vcd:weyl:groups} and \cref{prop:vcd:mpcg}, we have 

\begin{equation}\label{vcd:WF:nonorientable}
\vcd(WF)=\vcd(\calNgn[b_c]{g_F}{e_F+b_m}) =
\begin{cases}
b_c & \text{if } g_F=1 \text{ and } e_F+b\leq 2, \\
e_F+b_m+2b_c-2 & \text{if } g_F=1 \text{ and } e_F+b\geq 3, \\
e_F+b_m+2b_c & \text{if } g_F=2, \\
2g_F-5 & \text{if } g_F\geq 3 \text{ and } e_F+b=0,\\
2g_F+e_F+b_m+2b_c-4 & \text{if } g_F\geq 3 \text{ and } e_F+b\geq 1.
\end{cases}
\end{equation}

\begin{theorem}\label{thm:vcd:nonorientable:lambda:inequality:easycases}
Let $F$ be a finite subgroup of $\calN_g$ such that  $O_F$ is non-orientable. Assume that $g_F\geq 2$. Then, for all $g\geq 4$
\[\vcd(WF)+\lambda(F)\leq \vcd(\calN_g).\]
\end{theorem}
\begin{proof}
First note that the statement is trivially true if $F=1$. From now on we will assume $F\neq 1$.

By \cref{prop:vcd:geq:hurwitz},
\[
\vcdNgoverF\geq 2 g_F-4+e_F+b_m+2b_c+b_m+\frac{c_F}{2}=(\star)
\]
Once we prove that $\vcd(WF)\geq 1$ and $(\star)\geq \vcd(WF)$, using \cref{prop:vcd:inequalities} (\ref{prop:vcd:inequalities:zero}),  the proof will be done.
We will proceed by cases.

\begin{description}
\item[Case 1: $g_F\geq3$ and $e_F+b\geq 1$] By \eqref{vcd:WF:nonorientable} we get $(\star)=\vcd(WF)+\frac{c_F}{2}+b_m\geq \vcd(WF)$ since $\frac{c_F}{2}+b_m\geq 0$. On the other hand it is clear from \eqref{vcd:WF:nonorientable} that $\vcd(WF)\geq 1$.

\item[Case 2: $g_F\geq3$ and $e_F+b=0$]  By \eqref{vcd:WF:nonorientable} we get $(\star)=2g_F-4\geq 2g_F-5=\vcd(WF)$. In this case $\vcd(WF)\geq 1$ since we are assuming $g_F\geq 3$.

\item[Case 3: $g_F=2$ and $e_F+b_m+2b_c\geq1$] By \eqref{vcd:WF:nonorientable} we get  we get $(\star)=\vcd(WF)+\frac{c_F}{2}+b_m\geq \vcd(WF)$ since $\frac{c_F}{2}+b_m \geq 0$. Clearly in this case $\vcd(WF)\geq 1$.

\item[Case 4: $g_F=2$ and $e_F+b_m+2b_c = 0$] This case is impossible. In fact, we would have $|F|\chi(N_g)=\chi(O_F)=0$, but the left hand side is positive as $g\geq 3$. \end{description}
\end{proof}

\begin{theorem}\label{thm:vcd:nonorientable:lambda:inequality:hardcases}
Let $F$ be a finite subgroup of $\calN_g$ such that  $O_F$ is non-orientable. Assume that $g_F= 1$. Then, for all $g\geq 5$
\[\vcd(WF)+\lambda(F)\leq \vcd(\calN_g).\]
\end{theorem}
\begin{proof}
The statement is trivially true if $F=1$. From now on we will assume $F\neq 1$.
By \cref{prop:vcd:geq:hurwitz},
\[
\vcdNgoverF\geq  -2+e_F+b_m+2b_c+b_m+\frac{c_F}{2}=(\star)
\]

If $e_F+b\geq 3$, then by \eqref{vcd:WF:nonorientable}, we get $(\star)=-2+b+e_F+b_c+b_m+\frac{c_F}{2}\geq \vcd(WF)$ since $b_m+\frac{c_F}{2}\geq 0$. It is also clear that $\vcd(WF)\geq 1$. Therefore the conclusion follows in this case from \cref{prop:vcd:inequalities} \eqref{prop:vcd:inequalities:zero}.

Now we have to deal with the case $e_F+b\leq 2$. In this case $\vcd(WF)=b_c$ by \eqref{vcd:WF:nonorientable}. We will split this case into three subcases.

\begin{description}
    \item[Case 1: $g_F=1$ and $e_F+b=0$] This case is impossible. In fact, we would have $\chi(N_g)=|F|\chi(O_F)=|F|$, but the left hand side is negative as $g\geq 3$.

    \item[Case 2: $g_F=1$ and $e_F+b=2$] In this case either $b_c \neq 0$ or $b_c = 0$.
    
    Assume that $b_c \neq 0$. Then $(\star)=b_c+\frac{c_F}{2}+b_m\geq \vcd(WF) \geq 1$. Now the claim follows from \cref{prop:vcd:inequalities} \eqref{prop:vcd:inequalities:zero}. 
    
    Assume that $b_c = 0$, then we have that $\vcd(WF)=0$ and therefore we need to prove that 
    \[\vcd(\calN_g) \geq \lambda(F).\]
    
    If $b_m\neq 0$ then $e_F = 1 $ or $0$. In both cases, it follows from $(\star)$ that 
    \[
    \vcdNgoverF\geq  1
    \]
    and therefore
    \begin{align*}
        \vcd(\calN_g)\geq |F|-1 \geq  \lambda(F)
    \end{align*}
    
    If $b_m = 0$ then $e_F =2$. As $g \geq 3$, by \eqref{eq:riemann:hurwitz} we have at least one elliptic point with order greater than or equal to three. Then   
    \begin{align*}
    \vcdNgoverF = -2\chi(O_F) \geq \frac{1}{3}
   \end{align*}
    and therefore 
    \begin{align*}
        \vcd(\calN_g)\geq \frac{|F|}{3}-1 > \log_2(|F|) \geq \lambda(F)
    \end{align*}
    where the second inequality is true when $|F|>14$. Then $\lambda(F) \leq \vcd(\calN_g)$ when $|F|> 14$. If $|F| \leq 14$, then $\lambda(F)\leq 3$, thus $\lambda(F)\leq 3\leq \vcd(\calN_g)$ provided $g\geq 4$.

    \item[Case 3: $g_F=1$ and $e_F+b=1$] In this case either $e_F=0$ and $b=1$, or $e_F=1$ and $b=0$.
    
    Assume that $e_F=1$ and $b=0$. Then $\chi(O_F)=1-(1-\frac{1}{q})=\frac{1}{q}$. Therefore by \eqref{eq:riemann:hurwitz} we have $|F|\frac{1}{q}=2-g<0$ as $g\geq 4$. This is a contradiction, and this case is impossible.
    
    Assume that $e_F=0$ and $b=1$. We  have $2-g=|F|(-\frac{1}{2}C_F)$. Since $2-g<0$, then we conclude $C_F>0$ and therefore $c_F> 0$. This implies $\vcd(WF)=b_c=1$ and $b_m=0$. Next, note that $(\star)=\frac{c_F}{2}$. If $c_F\geq 2$, then $(\star)\geq 1=\vcd(WF)$, thus the claim follows from \cref{prop:vcd:inequalities} \eqref{prop:vcd:inequalities:zero}. If $c_F=1$, then 
    \begin{align*}
        \vcd(\calN_g)\geq \frac{|F|}{2}-1 > \log_2(|F|) \geq \lambda(F)
    \end{align*}
    where the first inequality comes from \cref{prop:vcd:geq:hurwitz}, the second one is true when $|F|> 8$  and the third one comes from \cref{lemma:lambda:inequalities}. We conclude that $\vcd(WF)+\lambda(F)=1+\lambda(F)\leq \vcd(\calN_G)$ when $|F|>8$. If $|F|\leq 8$, then $\lambda(F)\leq \log_2(|F|)\leq 3$, thus $\vcd(WF)+\lambda(F)\leq 4\leq \vcd(\calN_g)$ provided $g\geq 5$.
\end{description}
\end{proof}

\begin{remark}\label{remark:exceptional1}
Note that, in the proof of the previous theorem, the case $g=4$ was excluded only at the end of Case 3. Moreover, the only possible problem comes from the existence of a group $F$ of order 8 and with $\vcd(WF)=1$. In conclusion, we get $\vcd(WF)+\lambda(F)\leq 4=\vcd(\calN_4)+1$ for all finite subgroup $F$ of $\calN_4$ with $O_F$ non-orientable.
\end{remark}
\section{Verifying the hypothesis of \cref{thm:aramayona:martinezperez} when $O_F$ is orientable }\label{section:orientable:case}

Let $F$ be a finite subgroup of $\calN_g$ such that  $O_F$ is orientable. As a consequence of \cref{thm:vcd:weyl:groups} and \cref{prop:vcd:mpcg}, we have 

\begin{equation}\label{vcd:WF:orientable}
\vcd(WF)=\vcd(\calSgn[b_c]{g_F}{e_F+b_m}) =
\begin{cases}
b_c & \text{if } g_F=0 \text{ and } e_F+b\leq 2, \\
e_F+b_m+2b_c-3 & \text{if } g_F=0 \text{ and } e_F+b\geq 3, \\
1+b_c & \text{if } g_F=1 \text{ and } e_F+b= 0, \\
e_F+b_m+2b_c & \text{if } g_F=1 \text{ and } e_F+b\geq 1, \\
4g_F-5 & \text{if } g_F\geq 2 \text{ and } e_F+b=0,\\
4g_F+e_F+b_m+2b_c-4 & \text{if } g_F\geq 2 \text{ and } e_F+b\geq 1.
\end{cases}
\end{equation}

The proof of the following result is completely analogous to the proof of \cref{thm:vcd:nonorientable:lambda:inequality:easycases}, the details are left to the reader.

\begin{theorem}\label{thm:vcd:orientable:lambda:inequality:easycases}
Let $F$ be a finite subgroup of $\calN_g$ such that  $O_F$ is orientable. Assume that $g_F\geq 1$. Then, for all $g\geq 4$
\[\vcd(WF)+\lambda(F)\leq \vcd(\calN_g).\]
\end{theorem}

\begin{theorem}\label{thm:vcd:orientable:lambda:inequality:hardcases1}
Let $F$ be a finite subgroup of $\calN_g$ such that  $O_F$ is orientable. Assume that $g_F= 0$ and $e_F+b\geq 3$. Then, for all $g\geq 5$
\[\vcd(WF)+\lambda(F)\leq \vcd(\calN_g).\]
\end{theorem}
\begin{proof}
We may assume that $F$ is non trivial. 

We distinguish two cases, when $\vcd(WF)\geq 1$ and when $\vcd(WF)=0$.

\textbf{Case 1:} Suppose $\vcd(WF)\geq 1$.  By Proposition \ref{prop:vcd:geq:hurwitz}, we have 
\begin{align}
\vcdNgoverF&\geq -4+e_F+2b_m+2b_c+\frac{c_F}{2}  \label{eqwf1} \\
& = \vcd(WF) +b_m +\frac{c_F}{2}-1. \label{eqwf2}
\end{align}
Note that if $b_m\geq 1$ or $c_F\geq 2$, then $(\ref{eqwf2})$ is greater than or equal to $\vcd(WF)$. Using Proposition \ref{prop:vcd:inequalities} (1), we obtain, in this situation, the result for all $g\geq 4$. It remains to consider the cases: $(b_m, c_F)=(0,1)$ and $(b_m, c_F)=(0,0)$.

\textbf{Case 1 (a): } If $b_m=0$ and $c_F=1$, then $b=b_c=1$  and by hypothesis we have that $e_F\geq 2$. Note that $\vcd(WF)=e_F-1$, then 
\begin{align*}
 -4+e_F+2b_m+2b_c+\frac{c_F}{2}  &= -4+e_F+2+\frac{1}{2}\\
 &=e_F-\frac{3}{2}\\
 &=\vcd(WF) -\frac{1}{2}
\end{align*}
 from (\ref{eqwf1}) and  Proposition \ref{prop:vcd:inequalities} (2)  we conclude that $\vcd(WF)+\lambda(F)\leq \vcd(\calN_g)$ when $\vcd(WF)\geq 2$. On the other hand, if $\vcd(WF)=1$, then  $\vcd(WF)-\frac{1}{2}=\frac{1}{2}$, by (\ref{eqwf1}) and the above equality we have that 
 \begin{align}\label{eq3}
 \vcd(\calN_g)\geq \frac{1}{2}|F|-1.
  \end{align}   
 By Lemma \ref{lemma:lambda:inequalities} we have  
 \begin{align} \label{eq4b}
  \lambda(F)&\leq\log_2(|F|)  < \frac{1}{2}|F|-1
 \end{align}
where the last inequality holds when $|F|>8$. Combining (\ref{eq3}) and (\ref{eq4b}) we conclude that $\lambda(F)+\vcd(WF)\leq \vcd(\calN_g)$ when $|F|>8$. Finally, if $|F|\leq 8$, then $\lambda(F)\leq 3$, therefore $\vcd(WF)+\lambda(F)\leq 4 \leq \vcd(\calN_g)$ since $g\geq 5$. 
 
\textbf{Case 1 (b):}  If $b_m=0$  and $c_F=0$, then $b_c=0$, therefore $\vcd(WF)=e_F-3$. By Proposition \ref{prop:vcd:geq:hurwitz}, we have 
\begin{align*}
\vcdNgoverF&\geq -4+e_F+2b_m+2b_c+\frac{c_F}{2}  \label{eqwf1} \\
& =-4+e_F\\
&=\vcd(WF)-1.
\end{align*}
Thus by Proposition \ref{prop:vcd:inequalities} (3)  we conclude that $\lambda(F)+\vcd(WF)\leq \vcd(\calN_g)$ when $\vcd(WF)\geq 2$. Now suppose $\vcd(WF)=1$, then $e_F=4$.  If $q_i=2$  for $i=1,...,4$, where the $q_i$'s are the order of the elliptic points of $O_F$, then $-\chi(O_F)=-2+E_F=-2+\sum_{i=1}^4 (1-\frac{1}{q_i})=0$, which is not possible by \eqref{eq:riemann:hurwitz}.  In case that some $q_i\neq 2$, we have 
\begin{align*}
\vcdNgoverF &= -4 +2E_F\\
&= -4+2 \sum_{i=1}^4 \left(1-\frac{1}{q_i}\right)\\
&\geq -4 +2 \left(\frac{1}{2}+\frac{1}{2}+\frac{1}{2}+\frac{2}{3}\right)\\
&= \frac{1}{3},  
\end{align*}
 thus we have that $\lambda(F)\leq \log_2(|F|)<\frac{|F|}{3}-1\leq \vcd(\calN_g)$, where the second inequality holds when $|F|>14$. In case $|F|\leq 14$, we have that $\lambda(F)\leq 3$, therefore $\lambda(F)+\vcd(WF)\leq 4\leq \vcd(\calN_g)$ since $g\geq 5$.
 
 \textbf{Case 2:} If $\vcd(WF)=e_F+b_m+2b_c-3=0$, then $e_F+b_m+2b_c=3$, but $e_F+b\geq 3 $ by hypothesis, therefore $b_c=0$, which implies $c_F=0$.   Thus 
\begin{equation}\label{eqbmf}
 \vcdNgoverF =-4+2b_m+E_F 
\end{equation}
 
 Since $e_F+b_m=3$, below we will explore all possibilities for $e_F$ and $b_m$. 
 
 Note that the cases $(e_F,b_m)=(3,0)$  and $(e_F,b_m)=(2,1)$ are not possible because (\ref{eqbmf}) would be  negative. 
 
 If $e_F=1$ and $b_m=2$, then 
 \[
 -4+2b_m+E_F=E_F=1-\frac{1}{p}\geq\frac{1}{2}
 \]
 using (\ref{eqbmf})  and if $|F|\geq 8$, we have 
 \[
 \lambda(F)\leq \log_2(|F|)\leq  \frac{|F|}{2}-1\leq \vcd(\calN_g), 
 \]
 If $|F|<8$, then  $\lambda(F)\leq 3\leq \vcd(\calN_g)$ as $g\geq 4$.  
 
 Finally, if $e_F=0$ and $b_m=3$, using (\ref{eqbmf}) we obtain 
 \[ \lambda(F)\leq |F|\leq 2|F|-1=\vcd(\calN_g).
 \]

\end{proof}

\begin{remark}\label{remark:exceptional2}
Note that the previous theorem is valid when $g\geq 5$. The case $g=4$ was ruled  out in the following situations:

\begin{itemize}
    \item In Case 1(a) when we have $\vcd(WF)=1$ and $|F|=8$. This in this case $\vcd(WF)+\lambda(F)=4=\vcd(\calN_4)+1$.
    
    \item In Case 1(b) when we use \cref{prop:vcd:inequalities} \eqref{prop:vcd:inequalities:one}, and when $\vcd(WF)=1$ and $|F|\geq 14$. In the first situation we can use the moreover part of \cref{prop:vcd:inequalities} \eqref{prop:vcd:inequalities:one} to conclude $\vcd(WF)+\lambda(F)=4=\vcd(\calN_4)+1$, while in the second situation we have $\lambda(F)\leq 3$ hence we also have $\vcd(WF)+\lambda(F)=4=\vcd(\calN_4)+1$. 
\end{itemize}

\end{remark}
\begin{theorem}\label{thm:vcd:orientable:lambda:inequality:hardcases2}
Let $F$ be a finite subgroup of $\calN_g$ such that  $O_F$ is orientable. Assume that $g_F= 0$ and $e_F+b\leq 2$. Then $\vcd(WF)\geq 1$, and for all $g\geq 7$
\[\vcd(WF)+\lambda(F)\leq \vcd(\calN_g).\]
\end{theorem}
\begin{proof}
The statement is trivially true if $F=1$. From now on we will assume $F\neq 1$. In this case 
\[\vcdNgoverF = 2b-4+C_F+2E_F = (\star) \]
If $g>3$,  from \eqref{eq:riemann:hurwitz} and \eqref{eq:inequalities:eF:cF}  we get 
\[ b+e_F +\frac{c_F}{2} \geq b+E_F+\frac{1}{2}C_F > 2.\]
It follows that $e_F+b\neq 0$, and even more as $e_F+b\leq 2$ we have  $c_F\geq 1$, in particular $\vcd(WF)=b_c\geq 1$. Then, it remains to deal with the cases $e_F+b=1$ and $e_F+b=2$. In most of our cases, the proof will be reduced to verify for which values of $|F|$ the first and second inequalities in the following chain
\begin{equation}\label{eq:vcd:log:inequalities}
\vcd(\calN_g)\geq \frac{|F|}{n}-1 > \log_2(|F|) \geq \lambda(F)
\end{equation}
are true 
for certain $n\in\dbN$. The third inequality is always true by \cref{lemma:lambda:inequalities}.
\begin{description}
    \item[Case 1: $e_F+b=1$] As $c_F\geq 1$, we have $\vcd(WF)=b_c=1$ and $e_F+b_m=0$. By \eqref{eq:riemann:hurwitz} we have
    \[ -1+\frac{1}{2}C_F>0\]
    and therefore by \eqref{eq:inequalities:eF:cF} we get $c_F \geq C_F>2$.
    
    If $c_F=3$, then
    \[0<(\star)=-2+C_F=-1+\frac{1}{p_1}+\frac{1}{p_2}+\frac{1}{p_3}\]
    where $p_1$, $p_2$, $p_3$ are the orders of the elliptic points. By \cref{lemma:smallest:possibilities} we get  $(\star)\geq \frac{1}{42}$. Therefore we obtain \eqref{eq:vcd:log:inequalities} with $n=42$ which  is true when $|F|> 405$. We conclude that $\vcd(WF)+\lambda(F)=1+\lambda(F)\leq \vcd(\calN_G)$ when $|F|>405$. If $|F|\leq 405$, then $\lambda(F)\leq \log_2(|F|)< 9$, thus $\vcd(WF)+\lambda(F)\leq 9\leq \vcd(\calN_g)$ provided $g\geq 7$.
    
    If $c_F=4$,then
    \[0<(\star)=-2+C_F=-2+\frac{1}{p_1}+\frac{1}{p_2}+\frac{1}{p_3}+\frac{1}{p_4}\]
    where $p_1$, $p_2$, $p_3$, $p_4$ are the orders of the elliptic points. Since $p_i\geq 2$ and $p_1=p_2=p_3=p_4=2$ is impossible, the smallest possible value for $-2+C_F$ in this case is reached when $p_1=p_2=p_3=2$ and $p_4=3$, Hence $C_F\geq \frac{13}{6}$. Then $(\star)\geq \frac{1}{6}$. Therefore  we obtain \eqref{eq:vcd:log:inequalities} with $n=6$ which  is true when $|F|> 37$. Therefore $\vcd(WF)+\lambda(F)=1+\lambda(F)\leq \vcd(\calN_g)$ when $|F|>37$. If $|F|\leq 37$ we have $\lambda(F)\leq \log_2(|F|)< 6$, thus $\vcd(WF)+\lambda(F)\leq 6\leq \vcd(\calN_g)$ provided $g\geq 6$.
    
    If $c_F=5$, using an argument very similar to that in the previous paragraph we have that $C_F\geq \frac{5}{2}$  since the smallest possible value for $-3+C_F$ is reached when $p_1=\cdots=p_5=2$. Thus $(\star)\geq \frac{1}{2}$. Therefore we obtain  \eqref{eq:vcd:log:inequalities} with $n=2$ which is true when $|F|> 8$. We conclude that $\vcd(WF)+\lambda(F)=1+\lambda(F)\leq \vcd(\calN_g)$ when $|F|>8$. If $|F|\leq 8$, then $\lambda(F)\leq \log_2(|F|)\leq 3$, thus $\vcd(WF)+\lambda(F)\leq 4\leq \vcd(\calN_g)$ provided $g\geq 5$.
    
    If $c_F\geq 6$, by \cref{prop:vcd:geq:hurwitz} we have $(\star) \geq  1 = \vcd(WF)$, then the claim follows from \cref{prop:vcd:inequalities} \eqref{prop:vcd:inequalities:zero}.

    \item[Case 2: $e_F+b=2$] As $c_F \geq 1$ we have two possibilities $b=2$ or $b=1$.
    
    If $b=1$, as $b_c\geq 1$ we have that $\vcd(WF)=1$, $b_m=0$ and $e_F=1$, even more, from \eqref{eq:riemann:hurwitz} we have that
    \begin{equation}\label{eq:riemann:hurwitz:gf=0:aux}
    E_F +\frac{1}{2}C_F > 1
    \end{equation}
    Now we have different cases depending on $c_F$. If $c_F = 1$, then
    \[
    (\star)=1-\frac{1}{p}-\frac{2}{q}
    \]
    for $p,q\geq 2$ natural numbers. By \cref{lemma:smallest:possibilities}, we have $(\star)\geq \frac{1}{12}$. Therefore we have  \eqref{eq:vcd:log:inequalities} with $n=12$ which implies $1+\lambda(F) = \vcd(WF)+\lambda(F)\leq \vcd(\calN_g)$ provided $g\geq 6$. Similarly if $c_F=2$ or $c_F=3$, we can easily see that $(\star)\geq \frac{1}{6}$ or $(\star) \geq \frac{1}{2}$. Therefore we have  \eqref{eq:vcd:log:inequalities} with $n=6$ and $n=2$ respectively, then we have $1+\lambda(F) = \vcd(WF)+\lambda(F)\leq \vcd(\calN_g)$ provided $g\geq 6$. Finally if $c_F\geq 4$, from \cref{prop:vcd:geq:hurwitz} we have $(\star)\geq\frac{1}{2}c_F -1\geq 1= \vcd(WF)$, then the claim follows from \cref{prop:vcd:inequalities} \eqref{prop:vcd:inequalities:zero}.    
    
    If $b=2$, then $e_F = 0$ and from \eqref{eq:inequalities:eF:cF} we have $(\star)=C_F \geq \frac{c_F}{2}$. If $c_F=1$, then $\vcd(WF)=1$ and $(\star)\geq\frac{1}{2}$, therefore  we have \eqref{eq:vcd:log:inequalities} with $n=2$ and thus $1+\lambda(F) = \vcd(WF)+\lambda(F)\leq \vcd(\calN_g)$ provided $g\geq 5$.
    
    If $b=2$ and $c_F\geq 2$ we have $b_c=1$ or $b_c=2$. If $b_c=\vcd(WF)=1$, from $(\star)$ we get
    \[\vcdNgoverF \geq 1\]
    and the claim follows from \cref{prop:vcd:inequalities} \eqref{prop:vcd:inequalities:zero}. Similarly, if $b_c=\vcd(WF)=2$ the claim follows from \cref{prop:vcd:inequalities} \eqref{prop:vcd:inequalities:one}. 
\end{description}

\end{proof}

\begin{remark}\label{remark:exceptional3} Note that the previous theorem is valid only when $g\geq 7$. The cases $g=4,5$ where ruled out in the following situations:

\begin{itemize}
    \item In Case 1, where we actually have $\vcd(WF)=1$.
   \item In Case 2, where we either have $\vcd(WF)=1$, or $\vcd(WF)=2$ and we make use of \cref{prop:vcd:inequalities} \eqref{prop:vcd:inequalities:one}. In the latter case we conclude that $\vcd(WF)+\lambda(F)\leq 4$
\end{itemize}
We will deal with this situations in \cref{section:exceptional:cases}.
\end{remark}

\section{Proof of the main theorems}\label{section:proof:main:thm}

\subsection{The closed case}
\begin{proof}[Proof of \cref{thm:main}]
First note that $\calN_g$ for $g=1,2$ is finite \cite{Ham65,Lic63}, thus the claim follows since the three dimensions are zero for a finite group. Next $\vcd(\calN_3)=1$, hence by a well-known theorem of Stallings \cite{St68}, $\calN_3$ is virtually free and so it acts on a tree with finite stabilizers. Therefore $\gdfin(\calN_3)=1$ and the claim follows from \eqref{eq:inequalities:dimensions}.

Since $3\leq\vcd(\calN_g)\leq \cdfin(\calN_g)$ for $g\geq 4$, by \eqref{eq:inequalities:dimensions}, $\cdfin(\calN_g)=\gdfin(\calN_g)$. Thus we only have to prove $\vcd(\calN_g)=\cdfin(\calN_g)$.

The proof of the remaining cases is obtained through the verification of the hypothesis of \cref{thm:aramayona:martinezperez} for $\calN_g$, this is,  for every finite subgroup  $F$ of $\calN_g$, we want to verify
\begin{equation}\label{eq:vcdWF:lambdaF;vcdNg}
    \vcd(WF)+\lambda(F)\leq \vcd(\calN_g)
\end{equation}
We proceed by cases
\begin{enumerate}
    \item If $O_F$ is non-orientable, $g_F\geq 2$, then \eqref{eq:vcdWF:lambdaF;vcdNg} is true for all $g\geq 4$ by \cref{thm:vcd:nonorientable:lambda:inequality:easycases}.
    \item If $O_F$ is non-orientable, $g_F= 1$, then \eqref{eq:vcdWF:lambdaF;vcdNg} is true for all $g\geq 5$ by \cref{thm:vcd:nonorientable:lambda:inequality:hardcases}.
    
    \item If $O_F$ is orientable, $g_F\geq 1$ then \eqref{eq:vcdWF:lambdaF;vcdNg} is true for all $g\geq 4$ by \cref{thm:vcd:orientable:lambda:inequality:easycases}.
    \item If $O_F$ is orientable, $g_F=0$ then \eqref{eq:vcdWF:lambdaF;vcdNg} is true for all $g\geq 7$ by \cref{thm:vcd:orientable:lambda:inequality:hardcases1} and \cref{thm:vcd:orientable:lambda:inequality:hardcases2}.
\end{enumerate}
Now the claim follows for $g\geq 7$. 

The last remaining case is when $g=6$. In this case $\vcd(\calN_6)=7$. Note that the strategy used above cannot be carried out in this case only because the conclusion of \cref{thm:vcd:orientable:lambda:inequality:hardcases2} does not include $g=6$. Moreover, analyzing the proof of \cref{thm:vcd:orientable:lambda:inequality:hardcases2} we can see that the case $g=6$ only is excluded in ``Case 1: $e_F+b=1$'', where we have $\vcd(WF)=1$. Hence it is enough to verify \eqref{eq:vcdWF:lambdaF;vcdNg} in this situation.  By \cite{Conder15} we know that the largest finite group $F$ acting on $\calN_6$ has order $160=(2^5)(5)$, and therefore its length is at most 6. On the other hand $N_6$ does not admit the action of a group with order $128=2^7$. Hence every finite group acting on $N_6$ has length at most 6. Therefore, in the particular case we are dealing with
\[\vcd(WF)+\lambda(F)=1+\lambda(F)\leq 7= \vcd(\calN_6).\]
And this concludes the proof.
\end{proof}

\subsection{The exceptional cases: $g=4,5$}\label{section:exceptional:cases}\hfill

\vskip 10pt

In this section we deal with the cases $g=4,5$, that is, we prove \cref{thm:exeptional:cases}. 
To deal with these cases we make use of the data base \cite{Conder15}.


\begin{proof}[Proof of \cref{{thm:exeptional:cases}}]
The first inequality in both claims follows from \eqref{eq:inequalities:dimensions}.

Let us first work with $\calN_4$. Note that $\vcd(\calN_4)=3$. Let $F$ be a finite subgroup of $\calN_4$.
By \cref{remark:exceptional1} we obtain $\vcd(WF)+\lambda(F)\leq 4$  provided $O_F$ is non-orientable.

Assume now that $O_F$ is orientable. If $g_F\geq 1$ then $\vcd(WF)+\lambda(F)\leq 3$. By \cref{remark:exceptional2} we conclude that $\vcd(WF)+\lambda(F)\leq 4$ whenever $g_F=0$ and $e_F+b\geq 3$.

Finally, by \cref{remark:exceptional3} we only have to deal with $F$ such that $\vcd(WF)=1$. In \cite{Conder15} we see that $F$ has either order less than 12 or it has order 48, 24, 16, 12. In either case $\lambda(F)\leq 5$ since $\lambda(F)$ is bounded by the number of prime factors of $|F|$. Therefore, in this case, $\vcd(WF)+\lambda(F)=1+\lambda(F)\leq 1+5=6$. Now the second inequality in our claim follows from \cref{thm:aramayona:martinezperez}.

\vskip 10pt

Now we work with $\calN_5$.
By \cref{thm:vcd:nonorientable:lambda:inequality:easycases}, and \cref{thm:vcd:nonorientable:lambda:inequality:hardcases} we get $\vcd(WF)+\lambda(F)\leq \vcd(N_5)$ provided $O_F$ is non-orientable. By \cref{thm:vcd:orientable:lambda:inequality:easycases}, and \cref{thm:vcd:orientable:lambda:inequality:hardcases1} we get $\vcd(WF)+\lambda(F)\leq \vcd(N_5)$ provided $O_F$ is orientable and $g_F\geq 1$, or $g_F=0$ and $e_F+b\geq 3$.  If $g_F=0$ and $e_F+b\leq 2$, by \cref{remark:exceptional3} $\vcd(WF)+\lambda(F)\leq 4$ unless possibly when $\vcd(WF)=1$. By \cite{Conder15} the order of $F$ is either less than 12 or it has order 120, 72, 60, 36, 24, 20, 18, or 16. In either case $\lambda(F)\leq 5$ since $\lambda(F)$ is bounded by the number of prime factors in $|F|$. Therefore if $\vcd(WF)=1$ we get $\vcd(WF)+\lambda(F)\leq 1+5=6$. Hence $\cdfin(\calN_5)\leq 6$ by \cref{thm:aramayona:martinezperez}.

\end{proof}

\begin{remark}
Let $F$ be a subgroup of $\calN_4$ of order 48, which exists as showed in \cite{Conder15}. Since $48=(2^4)(3)$, then we know that there exists a subgroup $H$ of $F$ of order $2^4=16$. Hence $\lambda(H)=4$. Therefore $\lambda(F)=5$. On the other hand, we can see in \cite{Conder15} that $O_F$ has signature $(0; +; [-]; \{(2,4,6)\})$, thus $O_F$ is orientable, $g_F=0$ and $b_c=1$. Hence by \eqref{vcd:WF:orientable} we get $\vcd(F)=1$. In conclusion, for this group we get $\vcd(WF)+\lambda(F)=6$, hence the upper bound for $\vcd(\calN_4)$ given in \cref{thm:exeptional:cases} is the lowest one that can be obtained using \cref{thm:aramayona:martinezperez}.
\end{remark}

\begin{remark}
By \cite{Conder15}, we can find a finite group $F$ of order 120 that acts on $\calN_5$ such that $O_F$ is orientable, $g_F=0$, and it has exactly one boundary component with three corner points of orders $(2,4,5)$. Hence by \eqref{vcd:WF:orientable}, $\vcd(WF)=1$.

In \cite{Po20} Francesco Polizzi pointed out that this 120 element group $F$ is isomorphic to $S_5$, thus we have the chain $1<\dbZ/2<\dbZ/2\times \dbZ/2 < A_4<A_5<S_5$. Thus $\lambda(F)\geq5$. On the other hand, since $120=(2^ 3)(3)(5)$ we get $\lambda(F)=5$. Hence
$\vcd(WF)+\lambda(F)= 6.$ Therefore the upper bound for $\vcd(\calN_5)$ given in \cref{thm:exeptional:cases} is the lowest one that can be obtained using \cref{thm:aramayona:martinezperez}.
\end{remark}

\subsection{The case with  punctures and no boundary components}\hfill

\vskip 10pt

For $g\geq 3$ and $n\geq 1$ we have the following \textit{Birman short exact sequence} 
\begin{equation}\label{birman:exact:seq}
    1 \longrightarrow \pi_1(N_{g,n-1}) \longrightarrow P\calNgn{g}{n} \longrightarrow P\calNgn{g}{n-1} \longrightarrow 1,
\end{equation}
where $\pi_1(N_{g,n-1})$ is the fundamental group of the surface $N_g$ minus $n-1$ points. It can be deduced from \cite[Theorem~1]{Gra73} and \cite[Theorem~2.1]{Kor02}. For $g=2$ and $n\geq 2$ a Birman short exact sequence can be deduced from  \cite[Theorem~2, Proposition~2]{Gra73} and  \cite[Theorem~2.1]{Kor02}.

\begin{proof}[Proof of \cref{thm:main:with:punctures}]

To prove (1), we will proceed by cases.

\textbf{Case I: g=1.} If $n=1$ then $\calN_{1,1}$ is finite by \cite[Theorem~4.1]{Kor02}. From now on we assume $n\geq2$. From \cite[pp. 617]{Sc70} we have the following exact sequence 
\begin{equation}
1\to \mathbb{Z}/2 \to P_n(N_1) \to P\calN_{1,n}\to 1    
\end{equation}
where $P_n(N_1)$ denotes the $n$-th pure braid group of $N_1$.

We claim that $\gdfin (P_n(N_1)) =\gdfin (P\calN_{1,n})$. In fact, let $X$ be a model for $\underline{E} P\calN_{1,n}$, then the induced $P_n(N_1)$-action endows $X$ with the structure of a model for $\underline{E}P_n(N_1)$. Thus $\gdfin (P_n(N_1)) \geq\gdfin (P\calN_{1,n})$. On the other hand, if $Y$ is a model for $\underline{E}P_n(N_1)$, then the fixed point set $Y^{\dbZ/2}$ admits a natural action of the normalizer of $\dbZ/2$ in $\underline{E}P_n(N_1)$, and therefore, an action of the Weyl group of $\dbZ/2$, which in this case is isomorphic to $P\calN_{1,n}$. Moreover $Y^{\dbZ/2}$ is a model for $\underline{E} P\calN_{1,n}$ (see for instance \cite[Lemma~1.3]{Lu00}). Therefore $\gdfin (P_n(N_1)) \leq\gdfin (P\calN_{1,n})$ and the claim follows.

We will show that $\gdfin( P_n(N_1))=n-2$. By  the Fadell-Neuwirth short exact sequence  given in \cite{Van1966},  we have 
\begin{equation}
    1 \to \pi_1(N_{1,n})\to P_{n+1}(N_1)\to P_{n} (N_1)\to 1.
\end{equation}
Since $P_2(N_1)$ is finite (see \cite{GG2004}), then $\gdfin (P_2(N_1))=0$. From an inductive argument and using \cite[Theorem 5.16]{Lu05} we conclude that $\gdfin (P_n(N_1))\leq n-2=\vcd (P_n(N_1))$. Therefore  $\gdfin( P_n(N_1))= n-2$. 

\textbf{Case II: $g=2$.} We will use induction over $n$. For $n=1$, from \cite[Theorem~A.5]{MS06}, we know that $\calN_{2,1} = (\dbZ \rtimes \dbZ/2 ) \times \dbZ/2$, then $\calN_{2,1}$ is virtually free and infinite and therefore 
$$\gdfin (P\calN_{2,1}) = 1 = \vcd (P\calN_{2,1}).$$
Suppose the conclusion hold for $n-1$ with $n\geq 2$, applying  Theorem \ref{thm:main}, \cite[Theorem 5.16]{Lu05} and \cite[Lemma~4.4~(1)]{Gui2010} to the Birman short exact sequence (\ref{birman:exact:seq}), we have that 
$$\gdfin (P\calN_{2,n})\leq \gdfin (\pi_1(N_{2,n-1})) +\gdfin ( P\calN_{2,n-1}) = 1+ \vcd (P\calN_{2,n-1})=\vcd (P\calN_{2,n}).$$

\textbf{Case III: $g=3$ or $g\geq 6$. }We will use induction over $n$. For the first case, if $n=1$, by (\ref{birman:exact:seq})  we have
\begin{equation*}
    1 \longrightarrow \pi_1(N_{g}) \longrightarrow P\calNgn{g}{1} \longrightarrow P\calN_g \longrightarrow 1.
\end{equation*}
using Theorem \ref{thm:main}, \cite[Theorem 5.16]{Lu05}, and  a non-orientable version of \cite[Lemma 4.4 (2)]{Gui2010} (the proof is completely analogous),  we have that 
$$\gdfin (P\calN_{g,1})\leq \gdfin (\pi_1(N_{g}))+ \gdfin (P\calN_g) =2+\vcd (P\calN_g) = 2g-3=\vcd (P\calN_{g,1}).$$ 
Now suppose the conclusion holds for $n-1$. Again, applying  Theorem \ref{thm:main}, \cite[Theorem 5.16]{Lu05} and \cite[Lemma 4.4 (1)]{Gui2010} to the Birman short exact sequence (\ref{birman:exact:seq}), we have that 
$$\gdfin (P\calN_{g,n})\leq \gdfin
(\pi_1(N_{g,n-1})) +\gdfin ( P\calN_{g,n-1})
= 1+ \vcd (P\calN_{g,n-1})=\vcd (P\calN_{g,n}).$$
The second and third part of our statement can be proved as in Case III above using \cref{thm:exeptional:cases} instead of \cref{thm:main}. 
\end{proof}

\subsection{The case with at least one boundary component}

\begin{proof}[Proof of \cref{thm:vcd:gd:boundary}]
As $\calNgnb$ is torsion free when $b\geq 1$, it is clear that $\vcd (\calNgnb)=\cd (\calNgnb)=\cdfin  (\calNgnb) $. If $\cd(\calNgnb) \neq 2$, then $ \vcd (\calNgnb)=\cdfin (\calNgnb)=\gdfin(\calNgnb)$ (see the paragraph right below equation \eqref{eq:inequalities:dimensions}). By \cref{prop:vcd:mpcg}, the only cases when $\cd(\calNgnb) = 2$ are $\calN_{1,0}^2$, $\calN_{1,2}^1$ and $\calN_{2,0}^1$. It is known that $\calN_{1,0}^2\cong \dbZ\times\dbZ$ and $\calN_{2,0}^1 \cong \dbZ\rtimes\dbZ$ (see \cite[p.~141]{MS06} and \cite[Theorem~A.7]{MS06} respectively), therefore in these cases the proper cohomological and geometric dimension coincide, since in both cases $\mathbb R^2$ is a model for the classifying space for proper actions.

Finally, we deal with the case $\calN_{1,2}^1$.  We have the following short exact sequence, which appears in \cite[p.~262]{St10},
\begin{equation*}
    1\to \mathcal Z \to P^k\calN_{1,2}^1 \to P^{k+1}\calN_{1,3} \to 1
\end{equation*}
where $\mathcal Z\cong \dbZ$ is generated by the Dehn twist along the boundary component, and the second and third term of the sequence  are  finite index subgroups of the pure mapping class groups $P\calN_{1,2}^1$  and $P\calN_{1,3}$, respectively. Moreover, by \cite[Theorem~6.2]{MS06}, $\mathcal Z$ is the centre of $\calN^1_{1,2}$. Therefore $P^{k+1}\calN_{1,3}$ is a finite index subgroup of $\calN^1_{1,2}/\mathcal Z$. Note that, by \cref{prop:vcd:mpcg}, $\vcd(\calN_{1,3})=1$,  therefore $\calN_{1,3}$, $P^{k+1}\calN_{1,3}$, and $\calN^1_{1,2}/\mathcal Z$ are virtually free. In consequence $\gdfin(\calN^1_{1,2}/\mathcal Z)=1$. Applying \cite[Thm. 5.16]{Lu05} to the short exact sequence 
\[1\to \mathcal Z \to\calN^1_{1,2}\to \calN^1_{1,2}/\mathcal Z\to 1 \] we have that $\gdfin (\calN_{1,2}^1)\leq 2 $, therefore $\gdfin \calN_{1,2}^1=\cd \calN_{1,2}^1=2$. 
\end{proof}

\section{Questions and final remarks}\label{section:questions}

In this section we state several questions that are related to the scope of our main results.

In \cref{thm:main} we proved that the proper cohomological  dimension, the proper geometric dimension, and the virtual cohomological dimension of $\calN_g$ are equal for $g\neq 4,5$. We were not able to deal with $g=4,5$ due to the existence of finite subgroups in $\calN_4$ and $\calN_5$ of large length compared to the virtual cohomological dimension of the corresponding ambient group.
For $g=4,5$ we provide bounds for the proper cohomological dimension of $\calN_g$ in \cref{thm:exeptional:cases}. The following question is very natural.

\begin{question}
Is it true that $\gdfin(\calN_g)=\cdfin(\calN_g)=\vcd(\calN_g)$ for $g=4,5$?
\end{question}



The following question has to do with the  cases not covered in the statement of \cref{thm:main:with:punctures}.

\begin{question}
Is it true that $\gdfin (P\calN_{g,n})=\cdfin (P\calN_{g,n}) =\vcd (P\calN_{g,n})$ for all $n\geq 1$ and $g=4,5$?
\end{question}

In \cref{thm:main:with:punctures} we proved that the proper cohomological  dimension, the proper geometric dimension, and the virtual cohomological dimension of the \emph{pure} mapping class group $P\calN_{g,n}$ are equal for $g\neq 4,5$ and $n\geq 1$. In order to have an analogous statement  for the \emph{full} mapping class group $\calN_{g,n}$ it is enough to have a positive answer for the following question.

\begin{question}
Let $g\geq 1$ and $n\geq 0$. Is it true that $\cdfin(P\calN_{g,n})=\cdfin(\calN_{g,n})$?
\end{question}

\begin{remark}
By \cite{Sa06}, we know that $\cdfin(P\calN_{g,n}^b)=\cdfin(\calN_{g,n}^b)$ for $b\geq 1$ since both $\calN_{g,n}^b$ and $P\calN_{g,n}^b$ are torsion-free (see \cref{thm:main:with:boundaries}) and the former has finite index in the latter group.
\end{remark}

\bibliographystyle{alpha} 
\bibliography{myblib}
\end{document}